\documentclass[11pt]{article}
\usepackage{geometry}                
\usepackage{graphicx}
\usepackage{amssymb}
\usepackage{epstopdf}
\usepackage{amsthm}
\newtheorem{prop}{Proposition}
\newtheorem{result}[prop]{Result}
\newtheorem{thm}[prop]{Theorem}
\newtheorem{cor}[prop]{Corollary}
\newtheorem{const}[prop]{Construction}

\title{1-Overlap Cycles for Steiner Triple Systems}
\author{Victoria Horan\thanks{\texttt{vhoran@asu.edu}} and Glenn Hurlbert\thanks{\texttt{hurlbert@asu.edu}} \\ School of Mathematics and Statistics \\ Arizona State University \\ Tempe, AZ  85287 USA}

\begin{document}
\maketitle

\begin{abstract}
	A number of applications of Steiner triple systems (e.g. disk erasure codes) exist that require a special ordering of its blocks.  Universal cycles, introduced by Chung, Diaconis, and Graham in 1992, and Gray codes are examples of listing elements of a combinatorial family in a specific manner, and Godbole invented the following generalization of these in 2010.  1-overlap cycles require a set of strings to be ordered so that the last letter of one string is the first letter of the next.  In this paper, we prove the existence of 1-overlap cycles for automorphism free Steiner triple systems of each possible order.  Since Steiner triple systems have the property that each block can be represented uniquely by a pair of points, these 1-overlap cycles can be compressed by omitting non-overlap points to produce rank two universal cycles on such designs, expanding on the results of Dewar.
\end{abstract}



\section{Introduction}

Steiner triple systems, or $(v,3,1)$-designs, appear in many interesting applications.  They can be viewed as set systems, combinatorial designs, hypergraphs, or many other types of structures.  A Steiner triple system of order $v$, or STS($v$), is a pair $(X, \mathcal{B})$ with $|X|=v$, where $\mathcal{B}$ is a set of triples, or blocks, from $X$.  The set $\mathcal{B}$ has the property that every pair of points from $X$ appears in exactly one triple in $\mathcal{B}$.  It has been completely determined for which values $v$ such a set system can exist.  

\begin{thm}\label{STS}
	\emph{\cite{Kirkman}}  There exists an STS($v$) if and only if $v \equiv 1, 3 \pmod 6$.
\end{thm}

Several authors have considered ordering blocks in Steiner triple systems in specific ways.  For example, in \cite{Disks}, the authors consider such an ordering to construct erasure codes.  In \cite{Dewar}, Dewar uses a modified universal cycle structure to list the blocks and points within each block of these designs in an organized manner.  \textbf{Universal cycles}, or \textbf{ucycles}, are a type of cyclic Gray code in which a string $a_1a_2 \ldots a_n$ may follow string $b_1 b_2 \ldots b_n$ if and only if $a_{i+1} = b_i$ for all $i \in \{1,2, \ldots , n-1\}$.  That is, the two substrings $a_2 a_3 \ldots a_n$ and $b_1 b_2 \ldots b_{n-1}$ are identical \cite{CDG}.  We can think of this as an $n-1$ overlap between the strings.  

Because any two blocks in a Steiner triple system can share at most one point in common, finding a ucycle over the blocks of an STS is clearly impossible - they would need to overlap in two points.  To remedy this problem, Dewar introduces a modified ucycle structure.  A \textbf{rank two universal cycle} is a ucycle on a block design in which each block is represented by just two of its elements.  Since any pair of points appears in exactly one block of an STS, they completely identify a unique block in the triple system.  Dewar constructs rank two ucycles for the special class of cyclic Steiner triple systems.  A \textbf{cyclic} design has automorphism group containing the cyclic group of order $v$, isomorphic to $\mathbb{Z}_v$, as a subgroup.  Since this subgroup contains the automorphism $\pi: i \mapsto i+1 \pmod v$, this implies that we can partition the blocks into classes so that within one class each block can be obtained from any other by repeated applications of $\pi$ on the block elements.

\begin{thm}\label{Dewar}
	\emph{(\cite{Dewar}, p. 200)}  Every cyclic STS$(v)$ with $v \neq 3$ admits a ucycle of rank two.
\end{thm}

While this result allows us to write the list of blocks as a modified ucycle, it is not easy to recover the design from a given ucycle.  Given just two points of a block, the only way to recover the missing point is to have a lookup table at hand, which (depending on applications) may defeat the purpose of creating a compact listing.  For this reason, we consider overlap cycles.

Overlap cycles were first introduced in \cite{Godbole} for binary and $m$-ary strings.  To extend this concept to Steiner triple systems, let $(X, \mathcal{B})$ be an STS($v$).  An \textbf{$s$-overlap cycle} (or $s$-\textbf{ocycle}) on $(X, \mathcal{B})$ is an ordered listing of the blocks in $\mathcal{B}$ so that the last $s$ points in one block are the first $s$ points of its successor in the listing.  In the case of triple systems, a 2-ocycle is a ucycle, which as previously discussed cannot be formed on any STS.  Hence we consider 1-ocycles.  

When writing out ocycles, we can list the sequence fully or we can choose to omit points that do not appear as an overlap (\textbf{hidden} points).  When we omit the hidden points, we say that the cycle is written in \textbf{compressed form}.  Using this concept, we can view Dewar's rank two ucycles as compressed 1-ocycles and easily obtain the following corollary to Theorem \ref{Dewar}.

\begin{cor}\label{CDewar}
	Every cyclic STS$(v)$ with $v \neq 3$ admits a 1-ocycle.
\end{cor}

It is a well-known result that there exists a cyclic STS($v$) for every $v \equiv 1, 3 \pmod 6$ except $v = 9$ (See \cite{TripleSystems}, Theorem 7.3).  In order to further differentiate our results from Dewar's, we will consider a different class of Steiner triple systems, namely  automorphism free (AF) Steiner triple systems, and prove the following result using recursive constructions.  However, the constructions used herein may be utilized with various base cases for different (and perhaps not automorphism free) Steiner triple systems.

\begin{result}\label{main}
	For every $v \equiv 1, 3 \pmod 6$ with $v \geq 15$, there exists an AF STS$(v)$ with a 1-ocycle.
\end{result}

We also include two other direct constructions of 1-ocycles for a Steiner triple system of each order (Results \ref{OC3m6} and \ref{OC1m6}), as well as another recursive construction (Result \ref{OCPC}).  While Dewar constructs rank two ucycles for all cyclic designs of each order, these direct constructions may be a simpler method of finding an ocycle when any STS($v$) will do.

In this paper, we will begin with a review of some recursive constructions of automorphism free Steiner triple systems in Section 2 and show 1-ocycle constructions that correspond.  Some of the larger but necessary base cases for these constructions may be found in the appendix.  Section 3 discusses similar results for other STS constructions and their corresponding 1-ocycle constructions.  As a future direction, it would be interesting to consider these structures over other types of designs, such as Steiner quadruple systems (see \cite{OCSQS}).

\section{Constructions of AF Steiner Triple Systems and 1-Ocycles}


\subsection{Recursive Constructions of AF Steiner Triple Systems}

The first construction produces an AF STS($2v+1$) from an AF STS($v$).

\begin{const}\label{C2v+1}
	Given $(X, \mathcal{A})$, an STS($v$) with $v \geq 15$ and with $X$ identified with $\mathbb{Z}_v$, construct a new design $(Y, \mathcal{B})$ with points: $$Y = (\mathbb{Z}_2 \times \mathbb{Z}_v )\cup\{\infty\} $$ and blocks:
	\begin{enumerate}
		\item  $\{(1,a), (1,b), (1,c)\}$ with $\{a,b,c\} \in \mathcal{A}$,
		\item  $\left\{ (0,x), (0,y), \left(1, \frac{x+y}{2} \right) \right\}$ with $\{x,y\} \subset X$, and
		\item  $\{(0,x), (1,x), \infty\}$ with $x \in X$.
	\end{enumerate}
\end{const}

The following theorem from \cite{AFSTS} proves that Construction \ref{C2v+1} is correct.

\begin{thm}\label{2v+1}
	\emph{\cite{AFSTS}}.  If $(X, \mathcal{A})$ is an STS($v$), then $(Y, \mathcal{B})$ is an STS($2v+1$).  In particular, if $(X, \mathcal{A})$ is AF, then $(Y, \mathcal{B})$ is AF.  
\end{thm}

The second construction produces an AF STS($2v+7$) from an AF STS($v$).

\begin{const}\label{C2v+7}
	Given $(X, \mathcal{A})$, an STS($v$) with $v \geq 15$, and with $X$ identified with $\mathbb{Z}_v$, construct a new design $(Y, \mathcal{B})$ with points: $$Y = (\mathbb{Z}_2 \times \mathbb{Z}_v) \cup \{ \infty_i \mid |i| \leq 3\}.$$ Fix $(Z, \mathcal{C})$ as some STS(7) on the points $\{-3, -2, -1, 0, 1, 2, 3\}$.  The blocks in our new design are as follows:
	\begin{enumerate}
		\item $\{(1,i), (1,j), (1,k)\}$ with $\{i,j,k\} \in \mathcal{A}$,
		\item  $\{ \infty_i , \infty_j , \infty_k\}$ with $\{i,j,k\} \in \mathcal{C}$,
		\item  $\{(0,x), (0, x+2), (0,x+6)\}$ with $x \in \mathbb{Z}_v$,
		\item  $\{(0,x), (1,x+y), (0, x+2y)\}$ with $\{x, y\} \subset \mathbb{Z}_v$ and $|y|>3$, and
		\item  $\{\infty_i, (1,j), (0, i+j)\}$ with $|i| \leq 3$ and $j \in \mathbb{Z}_v$.
	\end{enumerate}
\end{const}

The following theorem from \cite{AFSTS} proves that Construction \ref{C2v+7} is correct.

\begin{thm}\label{2v+7}
	\emph{\cite{AFSTS}}.  If $(X, \mathcal{A})$ is an STS($v$), then $(Y, \mathcal{B})$ is an STS($2v+7$).  In particular, if $(X, \mathcal{A})$ is AF, then $(Y, \mathcal{B})$ is AF.  
\end{thm}

\subsection{Base Cases}\label{BaseC}

The recursive constructions given in the previous subsection require six base cases in order to construct recursively an AF STS($v$) for every $v \equiv 1, 3 \pmod 6$ with $v \geq 15$.  These base cases are STS($v$)'s for $v = 15, 19, 21,25,27,33$.  We also provide 1-ocycles for a non-cyclic STS($v$) for $v = 9,13$.  We include a 1-ocycle for the cyclic STS(7) as it is used in the second recursive construction.  See the appendix for all cases with $v \geq 19$.
\\
\\
$\mathbf{v=7:}$ We use the cyclic $(7,3,1)$-design and produce the 1-ocycle: $$\underline{2},1, \underline{0}, 3,
\underline{4}, 2,
\underline{5}, 0,
\underline{6},4,\underline{1},5,\underline{3},6,\underline{2}$$ or, since each
pair appears in exactly one block, we may omit the non-overlap points to write it in compressed form as:
$$(2,0,4,5,6,1,3,2).$$
  $\mathbf{v=9:}$  We use the non-cyclic design (from \cite{SmallSTS}) and produce the 1-ocycle: 
 $$\underline{0},1, \underline{2}, 8,
\underline{5}, 3,
\underline{4}, 1,
\underline{7},8,\underline{6},3,\underline{0},4,\underline{8},1,\underline{3},7,
\underline{2},4,\underline{6},1,\underline{5},7,\underline{0}$$ or in compressed form:
$$(0,2,5,4,7,6,4,8,3,2,6,5).$$
\\
\\
  $\mathbf{v=13:}$  We use the non-cyclic design (from \cite{AFSTS}) and produce the 1-ocycle:
  $$\underline{1}, 2, \underline{0}, 9, \underline{10}, 12, \underline{1}, 3, \underline{5}, 7, \underline{11}, 9, \underline{6}, 7, \underline{12}, 8, \underline{4}, 9, \underline{5}, 10, \underline{8}, 6, \underline{3}, 11, \underline{10}, 7, \underline{4},$$
		$$\underline{4}, 0, \underline{3}, 7, \underline{2}, 5, \underline{12}, 3, \underline{9}, 8, \underline{2}, 10, \underline{6}, 5, \underline{0}, 12, \underline{11}, 2, \underline{4}, 6, \underline{1}, 9, \underline{7}, 0, \underline{8}, 11, \underline{1}.$$
\\
\\
$\mathbf{v=15:}$  We use the AF design (from \cite{SmallSTS}) and produce the 1-ocycle:
$$\begin{array}{c|c|c|c|c}
		210 & 807 & 5b7 & da4 & b94 \\
		0a9 & 73c & 742 & 4e5 & 48c \\
		971 & ce1 & 2dc & 52a & c95 \\
		153 & 1db & cb0 & a7e & 58d \\
		304 & b82 & 0de & eb6 & d76 \\
		461 & 236 & e83 & 6ca & 689 \\
		1a8 & 605 & 39d & a3b & 9e2
	\end{array}$$
	
\subsection{Recursive Constructions of 1-Overlap Cycles}

\begin{result}\label{OC2v+1}
	If there exists an AF STS$(v)$ with a 1-ocycle, then there exists an AF STS$(2v+1)$ with a 1-ocycle when $v \geq 15$.
\end{result}
\begin{proof}
	Using Construction \ref{C2v+1}, we construct an overlap cycle for $(Y, \mathcal{B})$ as follows.   We will construct a 1-overlap cycle for triples of type (1), and then for the triples of types (2) and (3), and finally show that this sequence may be joined with the sequence for triples of type (1).
\\
\\
		\textbf{Step 1:  Triples of type (1):}  Let $O$ be a 1-ocycle on $\mathcal{A}$.  Define $\{1\} \oplus O$ to be the cycle obtained by preceding each point in $O$ to with a 1, i.e. each point becomes an ordered pair with first coordinate 1.  Then $\{1\} \oplus O$ is a 1-ocycle for the set of triples of type (1).
		\\
		\\
		\textbf{Step 2:  Triples of type (2):}  We first define the \textbf{difference} of the triple to be the smaller of $x-y$ and $y-x$ (modulo $v$).  Then we partition the set of triples of type (2) depending on their difference $d$.  This creates an equivalence relation on the set of triples of type (2).  We will construct 1-ocycles for each equivalence class separately. 
		\begin{description}
			\item[$\mathbf{d=1}$:]  We have the overlap cycle (in compressed form, with hidden elements removed): $$(0,0),(0,1),(0,2), \ldots , (0,v-1), (0,0).$$
			\item[$\mathbf{d \geq 3}$:]  We follow the same procedure as for $d=1$ by beginning with point $(0,0)$ and moving to point $(0, d)$, then $(0, 2d)$, and so on.  We note however that if $d \mid v$ the procedure will not produce a cycle that covers all triples.  However, when this happens we can repeat the process beginning with the first triple that remains unused.  In this manner, we will obtain several disjoint ocycles, and every triple of type (2) with the given difference $d$ will be covered by one of these cycles.
		\end{description}
		Note that for difference $d=1$, every point of type $(0,x)$ for $x \in X$ appears as an overlap point.  Thus for all overlap cycles associated with $d \geq 3$, we can join them to the cycle for $d=1$.  We reserve the triples corresponding to $d=2$ to include with the triples of type (3).
		\\
		\\
		\textbf{Step 3:  Triples of type (3):}  We construct an ocycle to include triples of type (2) with $d=2$ and triples of type (3) as follows.  First, connect pairs of triples of the form: $$\begin{array}{ccc} (1,x+1) & (0,x) & (0,x+2) \\ \\ & \hbox{and} &\\ \\ (0, x+2) & \infty & (1,x+2) \end{array}$$  Then we may use all of these pairs to form the ocycle: $$\underline{(1,1)}, (0,0), \underline{(0,2)}, \infty, \underline{(1,2)}, (0,1), \underline{(0,3)}, \infty, \underline{(1,3)}, \ldots , \underline{(1,0)}, (0,v-1), \underline{(0,1)}, \infty, \underline{(1,1)}.$$  This cycle accounts for $v$ of these pairs of triples, and since no triple is covered twice it must cover all triples of type (2) with $d=2$ and all triples of type (3).
	\\
	\\
	To connect all of the constructed cycles, we note that the cycle created in Step 3 contains the points $(0,x)$ and $(1,x)$ for every $x \in X$ as an overlap.  Thus we can connect the cycle created in Step 1 to this cycle, as well as the cycle created in Step 2.  This produces one long 1-ocycle that covers all triples.
\end{proof}

\begin{result}\label{OC2v+7}
	If there exists an AF STS$(v$) with a 1-ocycle, then there exists an AF STS$(2v+7)$ with a 1-ocycle, when $v \geq 15$.
\end{result}

\begin{proof}
	Using Construction \ref{C2v+7}, we will find ocycles for subsets of blocks, and show that they can be combined to form one long ocycle for the entire design.
	\\
	\\
		\textbf{Step 1:  Triples of type (1):}  Let $O$ be a 1-ocycle on $\mathcal{A}$.  Then $\{1\} \oplus \mathcal{A}$ also has a 1-ocycle, given by $\{1\} \oplus O$.
		\\
		\\
		\textbf{Step 2:  Triples of type (3):}  We construct one long cycle: $$\underline{(0,0)}, (0,6), \underline{(0,2)}, (0,8), \underline{(0,4)}, \ldots , \underline{(0, v-1)}, (0,5), \underline{(0,1)}, \ldots , \underline{(0,v-2)}, (0,4), \underline{(0,0)}.$$  Note that since $v$ must always be odd, we see the point $(0,x)$ for every $x \in \mathbb{Z}_v$ as an overlap point in this cycle.
		\\
		\\
		\textbf{Step 3:  Triples of type (4) with $|y|> 4$:}  We start by creating the ocycle (in compressed form, with hidden elements removed): $$(0,0), (0,2y), (0,4y), (0,6y), \ldots , (0,0).$$  This cycle contains all triples of type (4) associated with a particular $y$, since $v$ must be odd.  Note also that in each cycle, all of the points $(0,x)$ for every $x \in \mathbb{Z}_v$ appear as overlaps.  Thus we can connect all of these cycles for each choice of $y$ with $|y| > 4$. 
		\\
		\\
		\textbf{Step 4:  Triples of type (4) with $|y| = 4$ and type (5) with $i = -3$:}  We begin by pairing up blocks as follows so as to partition $\mathcal{B}$: $$\begin{array}{ccc} \{ (0,x), & (0,x+8), & (1,x+4)\} \\ \\ & \hbox{and} & \\ \\ \{(1,x+4), & \infty_{-3}, & (0,x+1)\} \end{array}$$  We can connect up these pairs in order, starting with the pair that begins $(0,0)$ and then moving to the pair that begins $(0,1)$, and so on.  We will eventually end with the pair starting $(0,v-1)$, which ends with the point $(0,0)$.  Thus we have an overlap cycle.  Note that in this cycle the points $(0,x)$ and $(1,x)$ appear as overlap points for every $x \in \mathbb{Z}_v$.
		\\
		\\
		\textbf{Step 5:  Triples of type (2):}  The triples of type (2) correspond to an STS(7).  We have shown in Section \ref{BaseC} that a 1-ocycle exists for the unique STS(7).  We will use the cycle from Step 4 to join the triples of type (2).  If we break the cycle from Step 4 between the blocks $$\begin{array}{cccc} \{(0,v-8), & (0,0), & (1,v-4)\} & \hbox{and} \\ \{(1,v-4), & \infty_{-3}, & (0,v-7)\} \end{array}$$ and also between the blocks $$\begin{array}{cccc} \{(1,3), & \infty_{-3}, & (0,0)\} & \hbox{and} \\ \{(0,0), & (0,8), & (1,4)\} \end{array}$$ then we now have two 1-overlap paths: $$\underline{(0,0)}, (0,8), \underline{(1,4)}, \ldots , \underline{(0,v-8)}, (0,0), \underline{(1,v-4)}$$ $$\hbox{and}$$ $$\underline{(1,v-4)}, \infty_{-3}, \underline{(0,v-7)}, \ldots , \underline{(1,3)}, \infty_{-3}, \underline{(0,0)}.$$  We can swap the order of the last two elements in the first path, and swap the order of the first two and the order of the last two elements in the second path two obtain the following two 1-ocycles: $$\underline{(0,0)}, (0,8), \underline{(1,4)}, \ldots , \underline{(0,v-8)}, (1,v-4), \underline{(0,0)}$$ $$\hbox{and}$$ $$\underline{\infty_{-3}}, (1,v-4), \underline{(0,v-7)}, \ldots , \underline{(1,3)}, (0,0), \underline{\infty_{-3}}.$$   Now we have $\infty_{-3}$ as an overlap point in the second cycle and so we can join this ocycle to the STS(7) ocycle (which contains every point $\infty_i$ as an overlap point).
		\\
		\\
		\textbf{Step 6:  Triples of type (5) with $i \neq -3$:}  We construct three separate ocycles as follows.  For $k \in \{-2,0,2\}$, construct the cycle: $$(0,0), (1,k), (0,1), (1,k+1), \ldots$$  When $k=-2$, this covers all triples of type $$\{(0,x), (1,x-2), \infty_2\} \hbox{ and } \{(0,x), (1,x-3), \infty_3\}.$$  When $k=0$, this covers all triples of type $$\{(0,x), (1,x), \infty_0\} \hbox{ and } \{(0,x), (1,x-1), \infty_1\}.$$  When $k=2$, this covers all triples of type $$\{(0,x), (1,x+2), \infty_{-2}\} \hbox{ and } \{(0,x), (1,x+1), \infty_{-1}\}.$$  These three cycles cover all triples of type (5) with $i \neq -3$.
		\\
		\\
	The cycles from Steps 2, 3, 5, and 6 all contain the point $(0,x)$ for every $x \in \mathbb{Z}_v $, and so can be connected.  The cycles from Steps 1, 5, and 6 all contain the point $(1,x)$ for every $x \in \mathbb{Z}_v \setminus \{v-4\}$, and so can be connected.  Since the cycle from Step 5 appears in both cases, these two long cycles can also be connected.  Thus, all triples are contained in one of the connected cycles, and so we have a 1-ocycle that covers all blocks.
\end{proof}

We are now ready to prove Result \ref{main}.
\begin{proof}[Proof of Result \ref{main}]
	We proceed by induction on $n$.  For $n=15,19,21,25,27,33$, we have shown ocycles in Section \ref{BaseC} and the appendix.
	
	For $n \geq 37$ and $n \equiv 1 \pmod {12}$, there exists $v \equiv 3 \pmod 6$ with $n = 2v+7$.  Note that $n \geq 37$ implies that $v \geq 15$.  Thus we use Result \ref{OC2v+7} to find the STS($n$).
	
	For $n \geq 39$ and $n \equiv 3 \pmod {12}$, there exists $v \equiv 1 \pmod 6$ with $n = 2v+1$.  Note that $n \geq 39$ implies that $v \geq 19$.  Thus we use Result \ref{OC2v+1} to find the STS($n$).
	
	For $n \geq 31$ and $n \equiv 7 \pmod {12}$, there exists $v \equiv 3 \pmod 6$ with $n = 2v+1$.  Note that $n \geq 31$ implies that $v \geq 15$, so we use Result \ref{OC2v+1} to find the STS($n$).
	
	For $n \geq 45$ and $n \equiv 9 \pmod {12}$, there exists $v \equiv 1 \pmod 6$ with $n = 2v+7$.  Note that $n \geq 45$ implies that $v \geq 19$, so we use Result \ref{OC2v+7} to find the STS($n$).
\end{proof}

\begin{cor}
	For every $n \geq 15$ with $n \equiv 1, 3\pmod 6$, there exists an AF STS$(n)$ with a rank two ucycle.
\end{cor}
\begin{proof}
	Using Result \ref{main}, we construct an AF STS($n$) with a 1-ocycle.  The 1-ocycle in compressed form is a rank two ucycle.
\end{proof}

\section{Other STS Constructions with Overlap Cycles}

In this section, we look at several other known constructions for Steiner triple systems, and show their corresponding 1-ocycle constructions.

\begin{const}\label{PC}
	\emph{(See \cite{TripleSystems}, p. 39 - Direct Product)}  Given $(X, \mathcal{A})$, an STS($u$) and $(Y, \mathcal{B})$, an STS($v$), identify $X$ with $\mathbb{Z}_u$ and $Y$ with $\mathbb{Z}_v$.  We construct a new STS($uv$) $=(Z, \mathcal{C})$ with points: $$Z = \mathbb{Z}_u \times \mathbb{Z}_v$$ and blocks:
	\begin{enumerate}
		\item  $\{(i,a), (i,b), (i,c)\}$ with $i \in \mathbb{Z}_u$ and $\{a,b,c\} \in \mathcal{B}$,
		\item  $\{(i,a), (j,a), (k,a)\}$ with $\{i,j,k\} \in \mathcal{A}$ and $a \in \mathbb{Z}_v$, and
		\item  $\{(i,a), (j,b), (k,c)\}$ with $\{i,j,k\} \in \mathcal{A}$ and $\{a,b,c\} \in \mathcal{B}$.
	\end{enumerate}
\end{const}

The following theorem (see \cite{TripleSystems}) proves that Construction \ref{PC} is correct.

\begin{thm}
	  If $(X, \mathcal{A})$ is an STS$(v)$ and $(Y, \mathcal{B})$ is an STS$(w)$, then $(Z, \mathcal{C})$ is an STS$(vw)$.
\end{thm}

An interesting consequence of the direct product is the following theorem.
\begin{thm}
	\emph{(See \cite{TripleSystems}, Lemma 7.12)}  The automorphism group of the direct product of two triple systems is the direct product of their automorphism groups.
\end{thm}

This theorem implies another method for constructing AF Steiner triple systems that admit ocycles.  Beginning with two AF Steiner triple systems with corresponding 1-ocycles, we can use the following result to construct a 1-ocycle on their direct product.

\begin{result}\label{OCPC}
	If there exists an STS$(u)$ with a 1-overlap cycle and an STS$(v)$ with a 1-overlap cycle, then there exists an STS$(uv)$ with a 1-ocycle.
\end{result}
\begin{proof}
	Let $(X, \mathcal{A})$ be an STS($u$) and $(Y, \mathcal{B})$ be an STS($v$) that admit 1-ocycles $O(\mathcal{A})$ and $O(\mathcal{B})$, respectively.  We construct an STS($uv$) using Construction \ref{PC} (the direct product).  For each $i \in \mathbb{Z}_u$, we have a 1-ocycle covering the triples of type 1, namely $i \oplus O(\mathcal{B})$.  Similarly, for each $a \in \mathbb{Z}_v$, we have a 1-ocycle covering the triples of type 2:  $O(\mathcal{A}) \oplus a$.  Lastly, for each $A=\{i,j,k\} \in \mathcal{A}$ and each $B =\{a,b,c\} \in \mathcal{B}$, we can construct the following 1-ocycle: 
	$$\begin{array}{ccc}
		\{(i,a), & (j,b), & (k,c)\} \\
		\{(k,c), & (j,a), & (i,b)\} \\
		\{(i,b) ,& (j,c), & (k,a)\} \\
		\{(k,a) ,& (j,b), & (i,c)\} \\
		\{(i,c) ,& (j,a), & (k,b)\} \\
		\{(k,b) ,& (j,c), & (i,a)\}
	\end{array}$$
	To connect cycles covering triples of types (1) and (2), we connect wherever possible.  Starting with $0 \oplus O(\mathcal{B})$, we connect all cycles over triples of type (2).  Then, starting with an arbitrary, already connected, cycle over triples of type (2), we repeat the process by adding cycles over triples of type (1) wherever possible.  We continue this process of extending our cycle until we no longer are able to add any more cycles.
	
	We will always be able to continue to connect cycles, except when all cycles are connected, or:
	\begin{enumerate}
		\item  there exists $i \in \mathbb{Z}_u$ that never appears as an overlap point in $O(\mathcal{A})$, and/or,
		\item  there exists $a \in \mathbb{Z}_v$ that never appears as an overlap point in $O(\mathcal{B})$.
	\end{enumerate}
	If we have both cases, then we choose a block $A \in \mathcal{A}$ containing $i$, say $A = \{i,j,k\}$, and a block $B \in \mathcal{B}$ containing $a$, i.e. $B = \{a,b,c\}$.  Then we arrange the cycle covering the triples from $A \times B$ to begin with $(i,a)$.  Since two points of $A$ must appear as overlap points in $O(\mathcal{A})$ and $i$ is not one of them, we must have that $j$ and $k$ are overlap points in $O(\mathcal{A})$.  Similarly, $b$ and $c$ must be overlap points in $O(\mathcal{B})$.  Thus we can connect the cycle for $A \times B$ to the cycles $k \oplus O(\mathcal{B})$ (at point $(k,c)$) and $O(\mathcal{A}) \oplus c$ (at point $(k,c)$ as well).  Note that since each block can only contain one hidden element, this process will never use a block from $\mathcal{A}$ or $\mathcal{B}$ more than once.  If only case (1) or case (2) holds (but not both), this process is repeated with an arbitrary choice of block from $\mathcal{B}$.
\end{proof}

\begin{const}\label{3m6}
	\emph{(\cite{Bose}, Bose Construction)}  Suppose that $n \equiv 3 \pmod 6$;  then $n=3m$ for some $m$ odd.  The point set is made up of three copies of the integers modulo $m$.  Formally: $$X = \mathbb{Z}_3 \times \mathbb{Z}_m.$$  Blocks are of two types:
	\begin{enumerate}
		\item  $\{(a,i), (a,j), (a+1,k)\}$ with $i+j = 2k$ for each $a \in \mathbb{Z}_3$
		\item  $\{(0,i), (1,i), (2,i)\}$ for each $i \in \mathbb{Z}_m$
	\end{enumerate}	
\end{const}

The following theorem from \cite{Bose} proves that Construction \ref{3m6} is correct.

\begin{thm}
	\emph{\cite{Bose}}  If $n \equiv 3 \pmod 6$, there exists an STS$(n)$.
\end{thm}

\begin{result}\label{OC3m6}
	For $n \equiv 3 \pmod 6$ with $n>3$, there exists an STS$(n)$ that admits a 1-ocycle.
\end{result}
\begin{proof}
	We will use Construction \ref{3m6} to create an STS($n$), then construct 1-ocycles to cover each type of triples, and finally show how to connect them into one large cycle.  First, note that we have $m\geq 3$ since $n>3$, and so there exists at least three triples of each kind.
\\
\\
		\textbf{Step 1:  Triples of type (1) with $a=1$:}  Define the value $\min \{i-j ,j-i \}$, where subtraction is done in the group $\mathbb{Z}_m$, to be the \textbf{distance} for the triple $\{(1,i), (1,j), (2,k)\}$.  Partition the blocks of type (1) into classes so that  the blocks $\{(1,i), (1,j), (2,k)\}$ and $\{(1,r), (1,s), (2,t)\}$ are in the same class if and only if they have the same distance.  This defines an equivalence relation on the set of blocks of type (1) with $\frac{m-1}{2}$ different equivalence classes.  Create a cycle using the set of blocks having the form $\{(1,i), (1,i+1), (2,i+\frac{m-1}{2}+1)\}$ as shown below in compressed form:  $$(1,0) (1,1) (1,2) \cdots (1,m-1) (1,0)$$ in compressed form.  Create similar (possibly shorter) cycles using the blocks within each other equivalence class.  This creates at least one, if not several disjoint, cycles for each equivalence class.  Since the first cycle created (using blocks with distance 1) has every point $(1,i)$ as an overlap point, we can combine all of these cycles to make one long cycle.
		\\
		\\
		\textbf{Step 2:  Triples of type (1) with $a=2$:}
		Repeat as in Step 1.  We pay careful attention to attach the cycle for distance 2 blocks at the point $(2,0)$.  Note that this is possible since distance 2 also creates one long cycle covering the entire equivalence class, as $m$ must be odd.  Now we may be assured that the cycle corresponding to distance 2 does not have any cycles attached at the overlap point $(2,1)$ between the blocks $\{(2,m-1), (0,0), (2,1)\}$ and $\{(2,1), (0,2), (2,3)\}$.  Then, when we have combined all blocks of type (1) with $a=2$ to make a cycle, we convert the cycle to a string by cutting it between these two blocks and then reversing the order of the last two points.  In other words, we now have a string that begins with $(2,1) (0,2) (2,3)$ and ends with $(2,m-1) (2,1) (0,0)$.
		\\
		\\
		\textbf{Step 3:  Triples of type (2) and (1) with $a=0$:}  Repeat as in Step 1 excluding the equivalence class with distance 2.  For these excluded blocks, we partition the set of blocks of type (1) and (2) into sets of size two by grouping together: $$\{(0,i), (2,i), (1,i)\} \hbox{ and } \{(1,i), (0,i-1), (0,i+1)\}.$$  Clearly the blocks in each set of size two can form a 1-overlap string, and then we can combine each of these strings to obtain a 1-ocycle of the form: $$\underline{(0,1)} (2,1) \underline{(1,1)} (0,0) \underline{(0,2)}  \cdots \underline{(0,i)} (2,i) \underline{(1,i)} (0,i-1) \underline{(0,i+1)} \cdots \underline{(0,0)} (2,0) \underline{(1,0)} (0,m-1) \underline{(0,1)}$$ $$\hbox{or}$$ $$(0,1) (1,1) (0,2) (1,2) \cdots (0,i) (1,i) (0,i+1) (1,i+1) \cdots (0,0) (1,0) (0,1)$$ in compressed form.
		\\
		\\
		\textbf{Step 4:  Combining the triples from Step 1 and Step 3}  Since the cycles created in Step 1  and Step 3 both contain every point $(1,i)$ as an overlap point, we can combine these two cycles.  More importantly, we have a choice of where to combine the cycles, since we have at least two choices for an overlap point $(1,i)$.  We choose to combine the two cycles at an overlap point other than $(1,1)$.  Then, we can create a string from this cycle by cutting the cycle between the blocks $\{(0,1), (2,1), (1,1)\}$ and $\{(1,1), (0,0), (0,2)\}$ (from cycle from Step 3), and reversing the order of the first two and the last two elements.  In other words, we now have a string that begins with $\{(2,1), (0,1), (1,1)\}$ and ends with $\{(1,1), (0,2), (0,0)\}$.  
		\\
		\\
		To create our final 1-ocycle, we recall that our string from Step 2 also begins with the point $(2,1)$ and ends with the point $(0,0)$, and so we can combine these two strings into one large cycle by reversing the order of the string from Step 4.
\end{proof}

\begin{const}\label{1m6}
	\emph{(\cite{Skolem}, Skolem Construction)}  If $n \equiv 1 \pmod 6$, then $n=6t+1$ for some $t \in \mathbb{Z}$.  We define the point set as $$Y = (\mathbb{Z}_{2t} \times \mathbb{Z}_3) \cup \{ \infty\}.$$  Then we define three types of blocks:
	\begin{enumerate}
		\item  $ A_x = \{(x,0), (x,1), (x,2)\}$ for $0 \leq x \leq t-1$.
		\item  $B_{x,y,i} = \{(x,i), (y,i), (x \circ y, i+1)\}$ for each $x , y \in \mathbb{Z}_{2t}$ with $x<y$ and each $i \in \mathbb{Z}_3$, and where $x \circ y = \pi(x+y \pmod {2t})$ and $$\pi(z) = \left\{ \begin{array}{ll}  z/2, & \hbox{ if } z \hbox{ is even,} \\ (z+2t-1)/2, & \hbox{ if } z \hbox{ is odd.}  \end{array}\right.$$
		\item  $C_{x,i} =  \{\infty, (x+t,i), (x,i+1)\}$ for each $0 \leq x \leq t-1$ and $i \in \mathbb{Z}_3$.
	\end{enumerate}
\end{const}

The following theorem from \cite{Skolem} proves that Construction \ref{1m6} is correct.

\begin{thm}
	If $n \equiv 1 \pmod 6$, then there is an STS$(n)$.
\end{thm}

\begin{result}\label{OC1m6}
	For $n \equiv 1 \pmod 6$ with $n>1$, there exists an STS$(n)$ that admits a 1-overlap cycle.
\end{result}
\begin{proof}
	We will use Construction \ref{1m6} to construct an STS($n$), then show how to construct disjoint cycles for most triples of type (2), then disjoint cycles for triples of types (1) and (3), and finally show how to combine them to make one large 1-ocycle containing all triples.
\\
\\
		\textbf{Step 1:  Triples of type (2):}  The triples of type (2) can be partitioned based on the pair $\{(x,i), (y,i)\}$.  Similar to Result \ref{OC3m6}, we define the \textbf{distance} of the triple to be the smaller of $x-y$ and $y-x$ (modulo $2t$).  Then, we can partition the set of triples of type (2) into classes that share the same distance for each difference $k<t$.  Following the method from Result \ref{OC3m6} (Step 1), we can create disjoint cycles that contain all of these triples.  Note that the triples corresponding to distance 1 make one long cycle for each second coordinate.  This cycle is (in compressed form): $$(0,i) (1,i) (2,i) \ldots (2t-1,i) \hbox{ for } i \in \mathbb{Z}_3.$$  These cycles contain every point from $\mathbb{Z}_{2t} \times \{i\}$ as an overlap point, and so we can hook up all of them to make three long cycles - one for each $i \in \mathbb{Z}_3$.  These cycles cover all triples of type (2) except those with distance $t$. 
		\\
		\\
		\textbf{Step 2:  Triples of types (1), (2), (3):}  We begin by partitioning the triples of type (3) into classes that contain the following blocks: $$\{\infty, (x+t,0), (x,1)\}, \{\infty, (x+t,1), (x,2)\}, \{\infty, (x+t,2), (x,0)\}.$$  Note that no other triples of type (3) contain any points with $x$ or $x+t$ as a first coordinate.  This set of blocks has a corresponding triple of type (1): $$\{(x,0), (x,1), (x,2)\}.$$  It also has a corresponding triple of type (2) with distance $t$: $$\{(x,i),(x+t,i), (x \circ (x+t), i+1)\} \hbox{ for } i \in \mathbb{Z}_3.$$  Using these blocks and defining $y = x+t$, we create the following cycle: $$\begin{array}{ccccccccccccccc}  \underline{x2} & (x\circ y)0 & \underline{y2}  & x0 & \underline{\infty} & x1 & \underline{y0} & (x \circ y)1 & \underline{x0} & x2 & \underline{x1} & (x \circ y)2 & \underline{y1} & \infty & \underline{x2} \end{array}$$  $$\hbox{or}$$ $$ \begin{array}{cccccccccc} x2 & y2 & \infty & y0 & x0 & x1 & y1 & x2   \end{array}$$  in compressed form.  This creates a set of disjoint cycles that cover all of the remaining triples.
	\\
	\\
	To combine all of our cycles and create our final 1-ocycle, we note that the cycles from Step 2 each have at least one overlap point of the form $(x,1)$ with $x \in \mathbb{Z}_t$, and so we can hook these cycles all up to the cycle from Step 1 that corresponds to $i=1$.  Also, each of the cycles from Step 2 also have overlap points $(x,i)$ corresponding to $i=1,2$ and $x \in \mathbb{Z}_t$, and so we can connect the remaining two cycles from Step 1.
\end{proof}

We can now use the direct constructions to prove the existence of an STS($v$) that admits a 1-ocycle for every $v \equiv 1, 3 \pmod 6$.

\begin{thm}\label{easymain}
  For every $v \equiv 1, 3 \pmod 6$, there exists an STS($v$) that admits a 1-ocycle.
\end{thm}

\begin{proof}
	For $n \equiv 3 \pmod 6$ with $n \geq 7$, we apply Result \ref{OC3m6} to obtain the desired system.  For $n \equiv 1 \pmod 6$ with $n \geq 7$, we apply Result \ref{OC1m6} to obtain the desired system.
\end{proof}

\begin{cor}
	For every $n \geq 7$ with $n \equiv 1, 3\pmod 6$, there exists an STS$(n)$ with a rank two ucycle.
\end{cor}
\begin{proof}
	Using Theorem \ref{easymain}, we construct an STS($n$) with a 1-ocycle.  The 1-ocycle in compressed form is a rank two ucycle.
\end{proof}

\newpage

\begin{appendix}

\section{Appendix}
Included in this appendix are the necessary base cases for Constructions \ref{C2v+1} and \ref{C2v+7}.
\\
\\
$\mathbf{v=19:}$  We use the AF design (from \cite{STS19}) and produce the 1-ocycle:
$$\begin{array}{l|l|l|l|l|l}
		1,2,3 	&  6,15,8	& 15,17,10	& 16,11,14	& 19,1,18	& 17,2,19 \\
		3,5,6 	&  8,1,9	& 10,5,14	& 14,17,7	& 18,2,16	& 19,14,8 \\
		6,2,4 	&  9,2,11	& 14,6,9	&  7,15,9	& 16,3,19	& 8,16,7 \\
		4,10,13 &  11,5,13	& 9,19,13	& 9,12,18	& 19,15,4	& 7,4,3 \\
		13,1,12 &  13,8,17	& 13,18,7	& 18,15,11	& 4,8,12	& 3,17,18 \\
		12,2,14 &  17,9,5	& 7,12,11	& 11,1,10	& 12,15,3	& 18,8,5 \\
		14,1,15 &  5,12,19	& 11,4,17	& 10,2,8	& 3,13,14	& 5,4,1 \\
		15,13,2 &  19,11,6	& 17,12,6	& 8,11,3	& 14,18,4 \\
		2,5,7 	&  6,13,16	& 6,18,10	& 3,9,10	& 4,9,16\\
		7,1,6 	&  16,5,15	& 10,12,16	& 10,7,19	& 16,1,17 
	\end{array}$$
\\
\\
  $\mathbf{v=21:}$  We use the AF design (from \cite{AFSTS}) to produce the 1-ocycle: 
$$\begin{array}{c|c|c|c|c|c|c}
				00, \infty_1 , 11  &  05 , \infty_1 , 16			& 12 , 14 , 08 		& 16 , 18 , 03  		& 01 , 10 , 15 & 05 , 10 , 14 & 05,03,04     \\
				11 , \infty_0 , 01 & 16 , \infty_0 , 06 			&  08 , \infty_2 , 11	&  03 , \infty_2 , 15	& 15 , 12 , 18 & 14 , 13 , 06 &  04, 01, 07   \\
				01 , \infty_1 , 12 & 06 , \infty_1 , 17 			&  11 , 13 , 17 		&  15 , 17 , 02		& 18 , 10 , 02 & 06 , 11 , 15 & 07, 08, 06   \\
				12 , \infty_0 , 02 & 17 , \infty_0 , 07 			&   07 , \infty_2 , 10	& 02 , \infty_2 , 14	& 02 ,11 , 16  & 15 , 14 , 07 & 06, 03, 00   \\
				02 , \infty_1 , 13 & 07 , \infty_1 , 18 			& 10 , 12 , 06		&  14 , 16 , 01		& 16 , 13 , 10 & 07 , 12 , 16 & 00, 04, 08     \\
				13 , \infty_0 , 03 & 18 , \infty_0 , 08 			& 06 , \infty_2 , 18	&  01 , \infty_2 , 13	& 10 , 11 , 03 & 16 , 15 , 08 & 08, 01, 03    \\
				03 , \infty_1 , 14 & 08 , \infty_1 , 10 			&  18 , 11 , 05		& 13 , 15 , 00	 	& 03 , 17 , 12 & 08 , 13 , 17 & 03, 07, 02       \\
				14 , \infty_0 , 04 & 10 , 00 , \infty_0 			&  05 , \infty_2 , 17	& 00 , 18 , 14 		& 12 , 11 , 04 & 17 , 16 , 00 & 02, 04, 06 \\
				04 , \infty_1 , 15 &  \infty_0 , \infty_1 , \infty_2 	&  17 , 10 , 04		& 14 , 11 , 17	 	& 04 , 18 , 13 & 00,01, 02    & 06, 01, 05  \\
				15 , \infty_0 , 05 & \infty_2 , 00 , 12			&  04 , \infty_2 , 16	& 17 , 18 , 01		& 13 , 12 , 05 & 02, 08, 05   & 05, 07, 00
			\end{array}$$
$\mathbf{v=25:}$  We use the AF design (from \cite{AFSTS}) to produce the 1-ocycle:
$$\begin{array}{c|c|c|c|c|c|c}
	\infty_1, 00, 11 & 18, \infty_0, 08	& 06, \infty_3, 10 & 04, \infty_5, 11 	& 07, 08, 06		& 13, 16, 14		& 14, \infty_6, 06  \\
	11, \infty_0, 01 & 08, \infty_1, 10	& 10, \infty_2, 07 &  11, \infty_4, 05	& 06, 03, 00		& 14, 17, 15		&  06, 01, 15 \\
	01, \infty_1, 12 & 10, \infty_0, 00	& 07, \infty_3, 11 &  05, \infty_5, 12	& 00, 04, 08		& 15, 18, 16		& 15, \infty_6, 07 \\
	12, \infty_0, 02 & 00, \infty_3, 13	& 11, \infty_2, 08 & 12, \infty_4, 06 	& 08, 01, 03		& 16, 10, 17		& 07, 02, 16 \\
	02, \infty_1, 13 & 13, \infty_2, 01	& 08, \infty_3, 12 &  06, \infty_5, 13	& 03, 07, 02		& 17, 11, 18		& 16, \infty_6, 08 \\
	13, \infty_0, 03 & 01, \infty_3, 14	& 12, \infty_2, 00 & 13, \infty_4, 07 	& 02, 04, 06		& 18, 12, 10 		& 08, 03, 17 \\
	03, \infty_1, 14 & 14, \infty_2, 02	& 00, \infty_5, 16 & 07, \infty_5, 14 	& 06, 01, 05		& 10, \infty_6, 02	& 17, 00, \infty_6  \\
	14, \infty_0, 04 & 02, \infty_3, 15 	& 16, \infty_4, 01 & 14, \infty_4, 08 	& 05, 07, 00		& 02, 06, 11		&  \infty_6, \infty_2, \infty_0 \\
	04, \infty_1, 15 & 15, \infty_2, 03	& 01, \infty_5, 17 &  08, \infty_5, 15	& 00, 04, 18		& 11, \infty_6, 03	& \infty_0, \infty_3, \infty_1 \\
	15, \infty_0, 05 & 03, \infty_3, 16 	& 17, \infty_4, 02 & 15, \infty_4, 00 	& 18, \infty_6, 01	& 03, 07, 12		&  \infty_1, \infty_4, \infty_2 \\
	05, \infty_1, 16 & 16, \infty_2, 04	& 02, \infty_5, 18 & 00, 01, 02 		& 01, 05, 10		& 12, \infty_6, 04	&  \infty_2, \infty_5, \infty_3 \\
	16, \infty_0, 06 & 04, \infty_3, 17	& 18, \infty_4, 03 &  02, 08, 05		& 10, 13, 11 		& 04, 08, 13		& \infty_3, \infty_6, \infty_4 \\
	06, \infty_1, 17 & 17, \infty_2, 05	& 03, \infty_5, 10 & 05, 03, 04	 	& 11, 14, 12		& 13, \infty_6, 05	& \infty_4, \infty_0, \infty_5 \\
	17, \infty_0, 07 & 05, \infty_3, 18 	& 10, \infty_4, 04 & 04, 01, 07	 	& 12, 15, 13		& 05, 00, 14		&  \infty_5, \infty_6, \infty_1 \\
	07, \infty_1, 18 & 18, \infty_2, 06	&
\end{array}$$
\\
\\
$\mathbf{v=27:}$  We use the AF design (from \cite{AFSTS}) to produce the 1-ocycle:
$$\hspace{-7mm}\begin{array}{c|c|c|c|c|c}
	00, \infty, 10 	& 19, 18, 12		& 07, \infty, 17		& 04, 00, 12		& 01, 06, 1(10)		& 0(11), 03, 17		  \\
	10, 0(12), 01 	& 12, 1(10), 16		& 17, 06, 08		& 12, 0(12), 05		& 1(10), 05, 02		& 17, 02, 0(12) 		  \\
	01, \infty, 11 	& 16, 15, 10		& 08, \infty, 18		& 05, 01, 13		& 02, 07, 1(11)		& 0(12), 04, 18 		 \\
	11, 12, 10 	& 10, 1(12), 1(11)	& 18, 07, 09		& 13, 00, 06		& 1(11), 06, 03		& 18, 03, 00		 \\
	10, 19, 1(10)	& 1(11), 12, 14		& 09, \infty, 19		& 06, 02, 14		& 03, 08, 1(12)		& 00, 07, 01 \\
	1(10), 1(12), 11	& 14, 16, 11		& 19, 08, 0(10)		& 14, 01, 07		& 1(12), 07, 04		& 01, 08, 02 \\
	11, 13, 15		& 11, 19, 17		& 0(10), \infty, 1(10)	& 07, 03, 15		& 04, 09, 10		& 02, 09, 03  \\
	15, 17, 1(11)	& 17, 10, 18		& 1(10), 09, 0(11)	& 15, 02, 08		& 10, 08, 05		& 03, 0(10), 04 \\
	1(11), 19, 16	& 18, 1(11), 11		& 0(11), \infty, 1(11)	&  08, 04, 16 		& 05, 0(10), 11		& 04, 0(11), 05 \\
	16, 17, 1(12)	& 11, 00, 02		& 1(11), 0(10), 0(12)	& 16, 03, 09		& 11, 09, 06		& 05, 0(12), 06 \\
	1(12), 18, 14 	& 02, \infty, 12		& 0(12), \infty, 1(12)	& 09, 05, 17		& 06, 0(11), 12		& 06, 00, 07 \\
	14, 19, 15 	& 12, 01, 03		& 1(12), 0(11), 00	& 17, 04, 0(10)		& 12, 0(10), 07		& 07, 01, 08 \\
	15, 1(10), 18	& 03, \infty, 13		& 00, 09, 1(11)		& 0(10), 06, 18		& 07, 0(12), 13		& 08, 02, 09 \\
	18, 16, 13		& 13, 02, 04		& 1(11), 08, 01		& 18, 05, 0(11)		& 13, 0(11), 08		& 09, 03, 0(10) \\
	13, 1(11), 1(10)	& 04, \infty, 14		& 01, 0(10), 1(12)	& 0(11), 07, 19		& 08, 00, 14		& 0(10), 04, 0(11) \\
	1(10), 17, 14	& 14, 03, 05 		& 1(12), 09, 02		& 19, 06, 0(12)		& 14, 0(12), 09		& 0(11), 05, 0(12) \\
	14, 10, 13		& 05, \infty, 15		& 02, 0(11), 10		& 0(12), 08, 1(10)	& 09, 01, 15		& 0(12), 06, 00 \\
	13, 17, 12		& 15, 04, 06 		& 10, 0(10), 03 		& 1(10), 07, 00		& 15, 00, 0(10)		&  \\
	12, 15, 1(12)	& 06, \infty, 16		& 03, 0(12), 11		& 00, 05, 19		& 0(10), 02, 16		&  \\
	1(12), 13, 19	& 16, 05, 07		& 11, 0(11), 04		& 19, 04, 01		& 16, 01, 0(11)		& 
\end{array}$$

$\mathbf{v=33:}$  We use the AF design (from \cite{AFSTS}) to produce the 1-ocycle:
$$\begin{array}{c|c|c|c}
	13, \infty_1, 02 		& 1(13), 16, 0(14) 		& 1(13), 11, 08		& 03, 09, 0(13) \\
	02, \infty_2, 15 		& 0(14), \infty_0, 1(14)	& 08, 15, 1(14)		& 0(13), 0(10), 04\\
	15, \infty_1, 04 		& 1(14), 14, 19			& 1(14), 12, 09		& 04, 0(14), 05\\
	04, \infty_2, 17 		& 19, 1(10), 0(14)		& 09, 16, 10		& 05, 02, 0(10) \\
	17, \infty_1, 06 		& 0(14), 18, 1(12)		& 10, 0(10), 13		& 0(10), 07, 0(14) \\
	06, \infty_2, 19 		& 1(12), 15, 0(13)		& 13, 1(10), 0(11)	& 0(14), 0(11), 06\\
	19, \infty_1, 08 		& 0(13), 17, 1(11)		& 0(11), 15, 19		& 06, 0(12), 0(10) \\
	08, \infty_2, 1(11) 	& 1(11), 14, 0(12)		& 19, 12, 0(10)		& 0(10), 03, 0(11) \\
	1(11), \infty_1, 0(10) 	& 0(12), 16, 1(10)		& 0(10), 14, 18		& 0(11), 09, 04 \\
	0(10), \infty_2, 1(13) 	& 1(10), \infty_0, 0(10)	& 18, 11, 09		& 04, 08, 0(12) \\
	1(13), \infty_1, 0(12) 	& 0(10), 17, 11			& 09, 13, 17		& 0(12), 09, 05 \\
	0(12), \infty_2, 10 	& 11, 14, 0(11)			& 17, 10, 08		& 05, 08, 0(13) \\
	10, \infty_1, 0(14) 	& 0(11), \infty_0, 1(11) 	& 08, 12, 16		& 0(13), 07, 06 \\
	0(14), \infty_2, 12 	& 1(11), 11, 16			& 16, 1(14), 07		& 06, 08, 09 \\
	12, \infty_1, 01 		& 16, 17, 0(11)			& 07, 11, 15		& 09, 0(14), 02 \\
	01, \infty_2, 14 		& 0(11), 18, 12			& 15, 1(13), 06		& 02, 1(11), 10 \\
	14, \infty_1, 03 		& 12, 15, 0(12)			& 06, 10, 14		& 10, 0(13), 12\\
	03, \infty_2, 16 		& 0(12), 19, 13			& 14, 1(12), 05		& 12, 00, 14 \\
	16, \infty_1, 05 		& 13, 16, 0(13)			& 05, 1(14), 13		& 14, 02, 16 \\
	05, \infty_2, 18 		& 0(13), 1(10), 14		& 13, 1(11), 04		& 16, 04, 18 \\
	18, \infty_1, 07 		& 14, 17, 0(14)			& 04, 1(13), 12		& 18, 06, 1(10) \\
	07, \infty_2, 1(10) 	& 0(14), 1(11), 15		& 12, 1(10), 03		& 1(10), 08, 1(12) \\
	1(10), \infty_1, 09 	& 15, 18, 00			& 03, 1(12), 11		& 1(12), 1(14), 0(10) \\
	09, \infty_2, 1(12) 	& 00, 1(12), 16			& 11, 19, 02		& 0(10), 15, 16 \\
	1(12), \infty_1, 0(11) 	& 16, 19, 01			& 02, 01, 00		& 16, \infty_0, 06 \\
	0(11), \infty_2, 1(14) 	& 01, 1(13), 17			& 00, 0(10), 09		& 06, 11, 12 \\
	1(14), \infty_1, 0(13) 	& 17, 1(10), 02			& 09, 07, 01		& 12, \infty_0, 02 \\
	0(13), \infty_2, 11 	& 02, 1(14), 18			& 01, 05, 03		& 02, 1(12), 1(13) \\
	11,00, \infty_1 		& 18, 1(11), 03			& 03, 00, 04		& 1(13), \infty_0, 0(13) \\
	\infty_1, \infty_0, \infty_2 & 03, \infty_0, 13		& 04, 06, 01		& 0(13), 18, 19 \\
	\infty_2, 00,13		& 13, 18, 1(13)			& 01, 0(10), 08		& 19, \infty_0, 09 \\
	13, 01, 15			& 1(13), 1(14), 03		& 08, 00, 07		& 09, 14, 15 \\
	15, 03, 17			& 03, 10, 19			& 07, 03, 0(12)		& 15, \infty_0, 05 \\
	17, 05, 19			& 19, 1(12), 04			& 0(12), 0(14), 01	& 05, 10, 11 \\
	19, 07, 1(11)		& 04, 11, 1(10)			& 01, 0(13), 0(11)	& 11, \infty_0, 01 \\
	1(11), 09, 1(13)		& 1(10), 1(13), 05		& 0(11), 08, 02		& 01, 1(11), 1(12) \\
	1(13), 0(11), 10		& 05, 12, 1(11)			& 02, 03, 06		& 1(12), \infty_0, 0(12) \\
	10, 19, 01			& 1(11), 1(14), 06		& 06, 00, 05		& 0(12), 17, 18 \\
	01, 1(10), 1(14)		& 06, 13, 1(12)			& 05, 0(11), 07		& 18, \infty_0, 08 \\
	1(14), 17, 00		& 1(12), 10, 07			& 07, 04, 02		& 08, 13, 14 \\
	00, \infty_0, 10		& 07, \infty_0, 17		& 02, 0(13), 0(12)	& 14, \infty_0, 04 \\
	10, 15, 1(10)		& 17, 1(12), 12			& 0(12), 0(11), 00	& 04, 10, 1(14) \\
	1(10), 1(11), 00		& 12, 13, 07			& 00, 0(14), 0(13)	& 1(14), 0(12), 11\\
	00, 19, 1(13) 		& 07, 14, 1(13)			& 0(13), 08, 03 		& 11, 0(14), 13
\end{array}$$

\end{appendix}

\end{document}